\documentclass[12pt,twoside]{amsart}
\usepackage{amsmath, amsthm, amscd, amsfonts, amssymb, graphicx}
\usepackage[bookmarksnumbered, plainpages]{hyperref}
\usepackage [latin1]{inputenc}
\newtheorem{thm}{Theorem}[section]

\numberwithin{equation}{section}

\begin{document}

\title{\bf Left-invariant Riemann solitons of three-dimensional Lorentzian Lie groups}
\author{Yong Wang}

\thanks{{\scriptsize
\hskip -0.4 true cm \textit{2010 Mathematics Subject Classification:}
53C40; 53C42.
 }}

\maketitle
\begin{abstract}
 In this note, we completely classify left-invariant Riemann solitons on three-dimensional Lorentzian Lie groups.
\end{abstract}

\pagestyle{myheadings}
\markboth{\rightline {\scriptsize Wang}}
         {\leftline{\scriptsize Left-invariant Riemann solitons}}

\bigskip


\section{ Introduction}
\indent Riemann solitons are generalized fixed points of the Riemann flow. In
the context of contact geometry, Hirica and Udriste proved \cite{HU} that if a Sasakian manifold
admited a Riemann soliton with potential vector field pointwise collinear with the structure
vector field ¦Î, then it was a Sasakian space form. In \cite{BL}, Blaga and Latcu studied almost Riemann solitons and almost Ricci solitons in an $(\alpha,\beta)$-contact metric manifold satisfying some Ricci symmetry conditions, treating the case when
the potential vector field of the soliton was pointwise collinear with the structure
vector field. In \cite{Ca1},
 Calvaruso studied three-dimensional generalized Ricci solitons, both in Riemannian and Lorentzian settings. He determined their homogeneous models, classifying left-invariant generalized Ricci solitons on three-dimensional Lie groups. In \cite{BO}, Batat and Onda studied
algebraic Ricci solitons of three-dimensional Lorentzian Lie groups. They got a complete classification of algebraic Ricci solitons of three-dimensional Lorentzian Lie groups. In \cite{Ca2}, Calvaruso completely classify three-dimensional homogeneous manifolds equipped with Einstein-like metrics.
In \cite{Wa}, we classify affine Ricci solitons associated to canonical connections and Kobayashi-Nomizu connections and perturbed canonical connections and perturbed Kobayashi-Nomizu connections on three-dimensional Lorentzian Lie groups with
 some product structure. In this note, we completely classify left-invariant Riemann solitons on three-dimensional Lorentzian Lie groups.


\vskip 1 true cm

\section{ Left-invariant Riemann solitons of three-dimensional Lorentzian Lie groups}

\indent Three-dimensional Lorentzian Lie groups had been classified in \cite{Ca3,CP}(see Theorem 2.1 and Theorem 2.2 in \cite{BO}). Throughout this paper,
we shall by $\{G_i\}_{i=1,\cdots,7}$, denote the connected, simply connected three-dimensional Lie group equipped with a left-invariant Lorentzian metric $g$ and
having Lie algebra $\{\mathfrak{g}\}_{i=1,\cdots,7}$. Let $\nabla$ be the Levi-Civita connection of $G_i$ and $R$ its curvature tensor, taken with the convention
\begin{equation}
R(X,Y)Z=\nabla_X\nabla_YZ-\nabla_Y\nabla_XZ-\nabla_{[X,Y]}Z.
\end{equation}
Let $R(X,Y,Z,W)=-g(R(X,Y)Z,W)$.
Riemann solitons are defined by a smooth vector field and a real constant $\lambda$ which satisfy the following equation:
\begin{equation}
R+\frac{1}{2}L_Vg\wedge g=\frac{\lambda}{2}g\wedge g,
\end{equation}
where $L_Vg$ denotes the Lie derivative of $g$ and $\wedge$ is the Kulkarni-Nomizu product. Let $T_1$ and $T_2$ be two arbitrary $(0,2)$-tensors, then their Kulkarni-Nomizu product is defined by
\begin{align}
T_1\wedge T_2(X,Y,Z,W)&:=T_1(X,W)T_2(Y,Z)+T_1(Y,Z)T_2(X,W)\\\notag
&-T_1(X,Z)T_2(Y,W)-T_1(Y,W)T_2(X,Z),
\end{align}
for any $X,Y,Z,W\in \Gamma(TG_i)$, where $\Gamma(TG_i)$ denotes the set of all vector fields on $G_i$. By (2.2) and (2.3), we can express
the Riemann soliton as follows:
\begin{align}
&2R(X,Y,Z,W)+g(X,W)(L_Vg)(Y,Z)+g(Y,Z)(L_Vg)(X,W)\\\notag
&-g(X,Z)(L_Vg)(Y,W)-g(Y,W)(L_Vg)(X,Z)\\\notag
&=2\lambda[g(X,W)g(Y,Z)-g(X,Z)g(Y,W)].
\end{align}
For $G_i$, there exists a pseudo-orthonormal basis $\{e_1,e_2,e_3\}$ with $e_3$ timelike. Let $V=\lambda_1e_1+\lambda_2e_2+\lambda_3e_3$, where $\lambda_1,\lambda_2,\lambda_3$ are real numbers. Let $R_{ijkl}=R(e_i,e_j,e_k,e_l)$. Then $(G_i,V,g)$ is a left-invariant Riemann soliton if and only if
\begin{align}
\left\{\begin{array}{l}
2R_{1212}-(L_Vg)(e_2,e_2)-(L_Vg)(e_1,e_1)=-2\lambda,\\
2R_{1312}-(L_Vg)(e_2,e_3)=0,\\
2R_{2312}+(L_Vg)(e_1,e_3)=0,\\
2R_{1313}-(L_Vg)(e_3,e_3)+(L_Vg)(e_1,e_1)=2\lambda,\\
2R_{2313}+(L_Vg)(e_1,e_2)=0,\\
2R_{2323}-(L_Vg)(e_3,e_3)+(L_Vg)(e_2,e_2)=2\lambda.\\
\end{array}\right.
\end{align}
By Theorem 2.1 in \cite{BO}, we have for $G_1$, there exists a pseudo-orthonormal basis $\{e_1,e_2,e_3\}$ with $e_3$ timelike such that the Lie
algebra of $G_1$ satisfies
\begin{equation}
[e_1,e_2]=\alpha e_1-\beta e_3,~~[e_1,e_3]=-\alpha e_1-\beta e_2,~~[e_2,e_3]=\beta e_1+\alpha e_2+\alpha e_3,~~\alpha\neq 0.
\end{equation}
By (2.18) in \cite{Ca2}, we have for $G_1$
\begin{align}
&R_{1212}=-2\alpha^2-\frac{\beta^2}{4},~~R_{1313}=\frac{\beta^2}{4}-2\alpha^2,~~R_{2323}=\frac{\beta^2}{4},\\\notag
&R_{1213}=2\alpha^2,~~R_{1223}=-\alpha\beta,~~R_{1323}=\alpha\beta.
\end{align}
Let
\begin{align}
L_Vg=\left(\begin{array}{ccc}
(L_Vg)(e_1,e_1)&(L_Vg)(e_1,e_2)&(L_Vg)(e_1,e_3)\\
(L_Vg)(e_2,e_1)&(L_Vg)(e_2,e_2)&(L_Vg)(e_2,e_3)\\
(L_Vg)(e_3,e_1)&(L_Vg)(e_3,e_2)&(L_Vg)(e_3,e_3)
\end{array}\right).
\end{align}
By page 7 in \cite{Ca1}, we get for $G_1$,
\begin{align}
L_Vg=\left(\begin{array}{ccc}
2\alpha(\lambda_2-\lambda_3)&-\alpha\lambda_1&\alpha\lambda_1\\
-\alpha\lambda_1&2\alpha\lambda_3&-\alpha(\lambda_2+\lambda_3)\\
\alpha\lambda_1&-\alpha(\lambda_2+\lambda_3)&2\alpha\lambda_2
\end{array}\right).
\end{align}
By (2.5)(2.7)(2.9) and $\alpha\neq 0$, we get that
$(G_1,V,g)$ is a left-invariant Riemann soliton if and only if
\begin{align}
\left\{\begin{array}{l}
-2\alpha^2-\frac{\beta^2}{4}-\alpha\lambda_2=-\lambda,\\
4\alpha+\lambda_2+\lambda_3=0,\\
\lambda_1=2\beta,\\
\frac{\beta^2}{4}-2\alpha^2-\alpha\lambda_3=\lambda,\\
\frac{\beta^2}{2}-2\alpha\lambda_2+2\alpha\lambda_3=2\lambda.\\
\end{array}\right.
\end{align}
The first equation plusing the fourth equation in (2.10), we get $\lambda_2+\lambda_3+4\alpha=0$. By the fourth equation and the fifth equation in (2.10), we $\lambda_2-2\lambda_3-2\alpha=0$. Then $\lambda_2=\lambda_3=-2\alpha$. By the first equation in (2.10), we get $\lambda=\frac{\beta^2}{4}$. So we have
\begin{thm}
$(G_1,V,g)$ is a left-invariant Riemann soliton if and only if $\lambda_1=2\beta,~\lambda_2=-2\alpha,~\lambda_3=-2\alpha,~\lambda=\frac{\beta^2}{4}$.
\end{thm}
By Theorem 2.1 in \cite{BO}, we have for $G_2$, there exists a pseudo-orthonormal basis $\{e_1,e_2,e_3\}$ with $e_3$ timelike such that the Lie
algebra of $G_2$ satisfies
\begin{equation}
[e_1,e_2]=\gamma e_2-\beta e_3,~~[e_1,e_3]=-\beta e_2-\gamma e_3,~~[e_2,e_3]=\alpha e_1,~~\gamma\neq 0.
\end{equation}
By page 144 in \cite{BO}, we have for $G_2$
\begin{align}
&R_{1212}=-\gamma^2-\frac{\alpha^2}{4},~~R_{1313}=\frac{\alpha^2}{4}+\gamma^2,~~R_{2323}=-\gamma^2-\frac{3}{4}\alpha^2+\alpha\beta,\\\notag
&R_{1213}=\gamma(2\beta-\alpha),~~R_{1223}=0,~~R_{1323}=0.
\end{align}
By page 8 in \cite{Ca1}, we get for $G_2$ (we correct a misprint in \cite{Ca1}),
\begin{align}
L_Vg=\left(\begin{array}{ccc}
0&\gamma\lambda_2+(\alpha-\beta)\lambda_3&(-\alpha+\beta)\lambda_2+\gamma\lambda_3\\
\gamma\lambda_2+(\alpha-\beta)\lambda_3&-2\gamma\lambda_1&0\\
(-\alpha+\beta)\lambda_2+\gamma\lambda_3&0&-2\gamma\lambda_1
\end{array}\right).
\end{align}
By (2.5)(2.12)(2.13), we get that
$(G_2,V,g)$ is a left-invariant Riemann soliton if and only if
\begin{align}
\left\{\begin{array}{l}
-\gamma^2-\frac{\alpha^2}{4}+\gamma\lambda_1=-\lambda,\\
\gamma(2\beta-\alpha)=0,\\
(-\alpha+\beta)\lambda_2+\gamma\lambda_3=0,\\
\frac{\alpha^2}{4}+\gamma^2+\gamma\lambda_1=\lambda,\\
\gamma\lambda_2+(\alpha-\beta)\lambda_3=0,\\
-\gamma^2-\frac{3}{4}\alpha^2+\alpha\beta=\lambda.\\
\end{array}\right.
\end{align}
By the first equation and the fourth equation and $\gamma\neq 0$ in (2.14), we get $\lambda_1=0$ and $\lambda=\frac{\alpha^2}{4}+\gamma^2$.
By the second equation and the sixth equation in (2.14), we get $\lambda=-\frac{\alpha^2}{4}-\gamma^2$. Then $\gamma=0$ and this is a contradiction.
So
\begin{thm}
$(G_2,V,g)$ is not a left-invariant Riemann soliton.
\end{thm}
By Theorem 2.1 in \cite{BO}, we have for $G_3$, there exists a pseudo-orthonormal basis $\{e_1,e_2,e_3\}$ with $e_3$ timelike such that the Lie
algebra of $G_3$ satisfies
\begin{equation}
[e_1,e_2]=-\gamma e_3,~~[e_1,e_3]=-\beta e_2,~~[e_2,e_3]=\alpha e_1.
\end{equation}
By page 146 in \cite{BO}, we have for $G_3$
\begin{align}
&R_{1212}=-(a_1a_2+\gamma a_3),~~R_{1313}=a_1a_3+\beta a_2,~~R_{2323}=-(a_2a_3+\alpha a_1),\\\notag
&R_{1213}=0,~~R_{1223}=0,~~R_{1323}=0,
\end{align}
where
\begin{equation}
a_1=\frac{1}{2}(\alpha-\beta-\gamma),~~a_2=\frac{1}{2}(\alpha-\beta+\gamma),~~a_3=\frac{1}{2}(\alpha+\beta-\gamma).
\end{equation}
By page 9 in \cite{Ca1}, we get for $G_3$,
\begin{align}
L_Vg=\left(\begin{array}{ccc}
0&(\alpha-\beta)\lambda_3&(\gamma-\alpha)\lambda_2\\
(\alpha-\beta)\lambda_3&0&(\beta-\gamma)\lambda_1\\
(\gamma-\alpha)\lambda_2&(\beta-\gamma)\lambda_1&0
\end{array}\right).
\end{align}
By (2.5)(2.16)(2.18), we get that
$(G_3,V,g)$ is a left-invariant Riemann soliton if and only if
\begin{align}
\left\{\begin{array}{l}
a_1a_2+\gamma a_3=\lambda,\\
(\beta-\gamma)\lambda_1=0,\\
(\alpha-\gamma)\lambda_2=0,\\
(\alpha-\beta)\lambda_3=0,\\
a_1a_3+\beta a_2=\lambda,\\
a_2a_3+\alpha a_1=-\lambda.\\
\end{array}\right.
\end{align}
\begin{thm}
$(G_3,V,g)$ is a left-invariant Riemann soliton if and only if \\
(i)$\beta=\gamma$, $\alpha\neq \gamma$, $\lambda_2=\lambda_3=0$, $\alpha=0$, $\lambda=0$,\\
(ii)$\alpha=\beta=\gamma$, $\lambda=\frac{1}{4}\alpha^2$,\\
(iii)$\beta\neq \gamma$, $\alpha=\beta$, $\lambda_1=\lambda_2=0$, $\gamma=0$, $\lambda=0$,\\
(iv)$\beta\neq \gamma$, $\alpha=\gamma$, $\lambda_1=\lambda_3=0$, $\beta=0$, $\lambda=0$.\\
\end{thm}
\begin{proof}
By the first equation and the fifth equation in (2.19), we get $a_1(a_2-a_3)+\gamma a_3-\beta a_2=0$. By (2.17), then we get $(\alpha-\beta-\gamma)(\beta-\gamma)=0$. By the fifth equation and the sixth equation in (2.19), we get
$(\alpha+\beta-\gamma)(\alpha-\beta)=0$ and
\begin{align}
\left\{\begin{array}{l}
(\beta-\gamma)\lambda_1=0,\\
(\alpha-\gamma)\lambda_2=0,\\
(\alpha-\beta)\lambda_3=0,\\
(\alpha-\beta-\gamma)(\beta-\gamma)=0,\\
(\alpha+\beta-\gamma)(\alpha-\beta)=0,\\
\lambda=a_1a_2+\gamma a_3.
\end{array}\right.
\end{align}
Case 1) $\beta\neq \gamma,~~\alpha\neq \gamma,~~\alpha\neq \beta$. Then by the fourth equation and the fifth equation in (2.20), we get
$\alpha=\gamma$. This is a contradiction and there are no solutions.\\
Case 2) $\beta=\gamma$, $\alpha\neq \gamma$. Solving (2.20), we get the case (i).\\
Case 3) $\alpha=\beta=\gamma$. Solving (2.20), we get the case (ii).\\
Case 4) $\beta\neq \gamma$, $\alpha=\beta$. Solving (2.20), we get the case (iii).\\
Case 5) $\beta\neq \gamma$, $\alpha=\gamma$. Solving (2.20), we get the case (iv).\\
\end{proof}
By Theorem 2.1 in \cite{BO}, we have for $G_4$, there exists a pseudo-orthonormal basis $\{e_1,e_2,e_3\}$ with $e_3$ timelike such that the Lie
algebra of $G_4$ satisfies
\begin{align}
[e_1,e_2]=-e_2+(2\eta-\beta)e_3,~~\eta=1~{\rm or}-1,~~[e_1,e_3]=-\beta e_2+ e_3,~~[e_2,e_3]=\alpha e_1.
\end{align}
By (2.32) in \cite{Ca2}, we have for $G_4$
\begin{align}
&R_{1212}=(2\eta-\beta)b_3-b_1b_2-1,~~R_{1313}=b_1b_3+\beta b_2+1,~~R_{2323}=-(b_2b_3+\alpha b_1+1),\\\notag
&R_{1213}=2\eta-\beta+b_1+b_2,~~R_{1223}=0,~~R_{1323}=0,
\end{align}
where
\begin{equation}
b_1=\frac{\alpha}{2}+\eta-\beta,~~b_2=\frac{\alpha}{2}-\eta,~~b_3=\frac{\alpha}{2}+\eta.
\end{equation}
By page 11 in \cite{Ca1}, we get for $G_4$,
\begin{align}
L_Vg=\left(\begin{array}{ccc}
0&-\lambda_2+(\alpha-\beta)\lambda_3&(\beta-\alpha-2\eta)\lambda_2-\lambda_3\\
-\lambda_2+(\alpha-\beta)\lambda_3&2\lambda_1&2\eta\lambda_1\\
(\beta-\alpha-2\eta)\lambda_2-\lambda_3&2\eta\lambda_1&2\lambda_1
\end{array}\right).
\end{align}
By (2.5)(2.22)(2.24), we get that
$(G_4,V,g)$ is a left-invariant Riemann soliton if and only if
\begin{align}
\left\{\begin{array}{l}
(2\eta-\beta)b_3-b_1b_2-1-\lambda_1=-\lambda,\\
2\eta-\beta+b_1+b_2-\eta\lambda_1=0,\\
(\beta-\alpha-2\eta)\lambda_2-\lambda_3=0,\\
b_1b_3+\beta b_2+1-\lambda_1=\lambda,\\
-\lambda_2+(\alpha-\beta)\lambda_3=0,\\
-(b_2b_3+\alpha b_1+1)=\lambda.\\
\end{array}\right.
\end{align}
\begin{thm}
$(G_4,V,g)$ is a left-invariant Riemann soliton if and only if \\
(i)$\beta\neq \eta$, $\alpha=0$, $\lambda_1=2-2\eta\beta$, $\lambda_2=\lambda_3=0$, $\lambda=0$,\\
(ii)$\alpha-\beta+\eta=0$, $\lambda_2=-\eta\lambda_3$, $\lambda_1=1-\eta\beta$, $\lambda=\frac{\alpha^2}{4}$.\\
\end{thm}
\begin{proof}
The fourth equation minusing the first equation in (2.25), we get $b_1b_3+\beta b_2+1-(2\eta-\beta)b_3+b_1b_2+1=2\lambda$. By the sixth equation in (2.25), we get $\alpha(\alpha-\beta+\eta)=0$.\\
Case 1) $\alpha-\beta+\eta\neq 0$. Then $\alpha=0$, solving (2.25), we get case (i).\\
Case 2) $\alpha-\beta+\eta= 0$. Solving (2.25), we get case (ii).\\
\end{proof}
By Theorem 2.2 in \cite{BO}, we have for $G_5$, there exists a pseudo-orthonormal basis $\{e_1,e_2,e_3\}$ with $e_3$ timelike such that the Lie
algebra of $G_5$ satisfies
\begin{equation}
[e_1,e_2]=0,~~[e_1,e_3]=\alpha e_1+\beta e_2,~~[e_2,e_3]=\gamma e_1+\delta e_2,~~\alpha+\delta\neq 0,~~\alpha\gamma+\beta\delta=0.
\end{equation}
By (2.36) in \cite{Ca2}, we have for $G_5$
\begin{align}
&R_{1212}=\alpha\delta-\frac{(\beta+\gamma)^2}{4},~~R_{1313}=-\alpha^2-\frac{\beta(\beta+\gamma)}{2}-\frac{\beta^2-\gamma^2}{4},\\\notag
&R_{2323}=-\delta^2-\frac{\gamma(\beta+\gamma)}{2}+\frac{\beta^2-\gamma^2}{4},~~
R_{1213}=0,~~R_{1223}=0,~~R_{1323}=0.
\end{align}
By page 13 in \cite{Ca1}, we get for $G_5$,
\begin{align}
L_Vg=\left(\begin{array}{ccc}
2\alpha\lambda_3&(\beta+\gamma)\lambda_3&-\alpha\lambda_1-\gamma\lambda_2\\
(\beta+\gamma)\lambda_3&2\delta\lambda_3&-\beta\lambda_1-\delta\lambda_2\\
-\alpha\lambda_1-\gamma\lambda_2&-\beta\lambda_1-\delta\lambda_2&0
\end{array}\right).
\end{align}
By (2.5)(2.27)(2.28), we get that
$(G_5,V,g)$ is a left-invariant Riemann soliton if and only if
\begin{align}
\left\{\begin{array}{l}
\alpha\delta-\frac{(\beta+\gamma)^2}{4}-\delta\lambda_3-\alpha\lambda_3=-\lambda,\\
\beta\lambda_1+\delta\lambda_2=0,\\
\alpha\lambda_1+\gamma\lambda_2=0,\\
-\alpha^2-\frac{\beta(\beta+\gamma)}{2}-\frac{\beta^2-\gamma^2}{4}+\alpha\lambda_3=\lambda,\\
(\beta+\gamma)\lambda_3=0,\\
-\delta^2-\frac{\gamma(\beta+\gamma)}{2}+\frac{\beta^2-\gamma^2}{4}+\delta\lambda_3=\lambda.\\
\end{array}\right.
\end{align}
\begin{thm}
$(G_5,V,g)$ is a left-invariant Riemann soliton if and only if \\
(i)$\beta+\gamma=0$, $\beta\neq 0$, $\alpha=\delta$, $\alpha\neq 0$, $\lambda_1=\lambda_2=\lambda_3=0$, $\lambda=-\alpha^2$,\\
(ii)$\beta=\gamma=0$, $\alpha=\delta$, $\alpha\neq 0$, $\lambda_1=\lambda_2=\lambda_3=0$, $\lambda=-\alpha^2$.\\
\end{thm}
\begin{proof}
Case 1) $\beta+\gamma\neq 0$. Then $\lambda_3=0$. By the fourth equation and the sixth equation in (2.29), we get $\alpha^2-\delta^2+\beta^2-\gamma^2=0$. By the first equation and the fourth equation in (2.29), we get $\alpha^2+\beta^2+\beta\gamma-\alpha\delta=0$.\\
Case 1)-a)$\beta\gamma-\alpha\delta=0$. We get $\alpha=\beta=\gamma=\delta=0$. This is a contradiction.\\
Case 1)-b)$\beta\gamma-\alpha\delta\neq 0$. By the second equation and the third equation in (2.29), we get $\lambda_1=\lambda_2=0$.\\
Case 1)-b)-1)$\alpha=0$. Then $\delta\neq 0$ and $\beta=0$, then $\delta=\gamma=0$ by $\alpha^2-\delta^2+\beta^2-\gamma^2=0$. This is a contradiction.\\
Case 1)-b)-2)$\alpha\neq 0$. Then $\gamma=-\frac{\beta\delta}{\alpha}$. Then $\beta\neq 0$ and $\alpha\neq \delta$ by $\beta+\gamma\neq 0$.
By $\alpha+\delta\neq 0$, then $\alpha^2\neq \delta^2$. By $\alpha^2-\delta^2+\beta^2-\gamma^2=0$, we get $1+\frac{\beta^2}{\alpha^2}=0$. This
is a contradiction.\\
Case 2) $\beta+\gamma=0$. By (2.29), we have
\begin{align}
\left\{\begin{array}{l}
\alpha\delta-\delta\lambda_3-\alpha\lambda_3=-\lambda,\\
\beta\lambda_1+\delta\lambda_2=0,\\
\alpha\lambda_1+\gamma\lambda_2=0,\\
-\alpha^2+\alpha\lambda_3=\lambda,\\
-\delta^2+\delta\lambda_3=\lambda.\\
\end{array}\right.
\end{align}
By $\alpha\gamma+\beta\delta=0$, we have $\beta(\alpha-\delta)=0$.\\
Case 2)-a)$\beta\neq 0$. Then $\alpha=\delta$. Solving (2.30), we get the case (i).\\
Case 2)-b)$\beta=0$. Then $\gamma=0$. So $\delta\lambda_2=0$, $\alpha\lambda_1=0$.\\
Case 2)-b)-1) $\alpha\neq 0$, $\delta\neq 0$. Then $\lambda_1=\lambda_2=0$. Solving (2.30), we get the case (ii).\\
Case 2)-b)-2) $\alpha= 0$, $\delta\neq 0$. Solving (2.30), we get $\delta= 0$. This is a contradiction.\\
Case 2)-b)-3) $\alpha\neq 0$, $\delta= 0$. Solving (2.30), we get $\alpha= 0$. This is a contradiction.\\
\end{proof}
\indent By Theorem 2.2 in \cite{BO}, we have for $G_6$, there exists a pseudo-orthonormal basis $\{e_1,e_2,e_3\}$ with $e_3$ timelike such that the Lie
algebra of $G_6$ satisfies
\begin{equation}
[e_1,e_2]=\alpha e_2+\beta e_3,~~[e_1,e_3]=\gamma e_2+\delta e_3,~~[e_2,e_3]=0,~~\alpha+\delta\neq 0£¬~~\alpha\gamma-\beta\delta=0.
\end{equation}
By (2.40) in \cite{Ca2}, we have for $G_6$
\begin{align}
&R_{1212}=-\alpha^2+\frac{\beta(\beta-\gamma)}{2}+\frac{\beta^2-\gamma^2}{4},~~
R_{1313}=\delta^2+\frac{\gamma(\beta-\gamma)}{2}+\frac{\beta^2-\gamma^2}{4},\\\notag
&R_{2323}=\alpha\delta+\frac{(\beta-\gamma)^2}{4},~~
R_{1213}=0,~~R_{1223}=0,~~R_{1323}=0.
\end{align}
By page 14 in \cite{Ca1}, we get for $G_6$,
\begin{align}
L_Vg=\left(\begin{array}{ccc}
0&\alpha\lambda_2+\gamma\lambda_3&-\beta\lambda_2-\delta\lambda_3\\
\alpha\lambda_2+\gamma\lambda_3&-2\alpha\lambda_1&(\beta-\gamma)\lambda_1\\
-\beta\lambda_2-\delta\lambda_3&(\beta-\gamma)\lambda_1&2\delta\lambda_1
\end{array}\right).
\end{align}
By (2.5)(2.32)(2.33), we get that
$(G_6,V,g)$ is a left-invariant Riemann soliton if and only if
\begin{align}
\left\{\begin{array}{l}
-\alpha^2+\frac{\beta(\beta-\gamma)}{2}+\frac{\beta^2-\gamma^2}{4}+\alpha\lambda_1=-\lambda,\\
(\beta-\gamma)\lambda_1=0,\\
\beta\lambda_2+\delta\lambda_3=0,\\
\delta^2+\frac{\gamma(\beta-\gamma)}{2}+\frac{\beta^2-\gamma^2}{4}-\delta\lambda_1=\lambda,\\
\alpha\lambda_2+\gamma\lambda_3=0,\\
\alpha\delta+\frac{(\beta-\gamma)^2}{4}-\delta\lambda_1-\alpha\lambda_1=\lambda.\\
\end{array}\right.
\end{align}
\begin{thm}
$(G_6,V,g)$ is a left-invariant Riemann soliton if and only if \\
(i)$\beta\neq \gamma$, $\lambda_1=0$, $\alpha=\beta=0$, $\lambda=\frac{\gamma^2}{4}$, $\lambda_3=0$, $\delta^2=\gamma^2$,\\
(ii)$\beta\neq \gamma$, $\lambda_1=0$, $\alpha\neq 0$, $\alpha^2=\beta^2$, $\delta=\frac{\beta\gamma}{\alpha}$, $\lambda=\frac{(\beta+\gamma)^2}{4}$, $\lambda_2=-\frac{\gamma}{\alpha}\lambda_3$,\\
(iii)$\beta=\gamma$, $\beta\neq 0$, $\alpha=\delta$, $\alpha\neq 0$, $\lambda_1=\lambda_2=\lambda_3=0$, $\lambda=\alpha^2$,\\
(iv)$\lambda_3\neq 0$, $\lambda_2=-\frac{\delta}{\beta}\lambda_3$, $\alpha\neq 0$, $\beta\neq 0$, $\beta=\gamma$, $\alpha=\delta$, $\alpha^2=\beta^2$, $\lambda_1=0$, $\lambda=\alpha^2$,\\
(v)$\beta=\gamma=0$, $\alpha\neq 0$, $\delta\neq 0$, $\lambda_1=\lambda_2=\lambda_3=0$, $\alpha=\delta$, $\lambda=\alpha^2$.\\
\end{thm}
\begin{proof}
Case 1) $\beta-\gamma\neq 0$. Then $\lambda_1=0$. So by the first, the fourth and sixth equations in (2.34), we get
\begin{equation}
\delta^2-\alpha^2+\beta^2-\gamma^2=0,~~\alpha^2-\beta^2+\beta\gamma-\alpha\delta=0.
\end{equation}
Case 1)-a) $\beta\gamma-\alpha\delta=0$. So $\alpha^2=\beta^2$ and $\delta^2=\gamma^2$ by (2.35).\\
Case 1)-a)-1) $\alpha=0$. Solving (2.34), we get the case (i).\\
Case 1)-a)-2)$\alpha\neq 0$. Then $\delta=\frac{\beta\gamma}{\alpha}$. Solving (2.34), we get the case (ii).\\
Case 1)-b) $\beta\gamma-\alpha\delta\neq 0$. So $\lambda_2=\lambda_3=0$.\\
Case 1)-b)-1) $\alpha=0$. So $\delta\neq 0$ and $\beta=0$. This is a contradiction with $\beta\gamma-\alpha\delta\neq 0$.\\
Case 1)-b)-2) $\alpha\neq 0$. We get $\gamma=\frac{\beta\delta}{\alpha}$ and $\alpha^2=\beta^2$ by (2.35). Then $\beta\gamma-\alpha\delta=0$.
This is a contradiction.\\
Case 2) $\beta-\gamma= 0$. Then $\beta(\alpha-\delta)=0$. By (2.34), we have
\begin{align}
\left\{\begin{array}{l}
-\alpha^2+\alpha\lambda_1=-\lambda,\\
\beta\lambda_2+\delta\lambda_3=0,\\
\delta^2-\delta\lambda_1=\lambda,\\
\alpha\lambda_2+\gamma\lambda_3=0,\\
\alpha\delta-\delta\lambda_1-\alpha\lambda_1=\lambda.\\
\end{array}\right.
\end{align}
Case 2)-a) $\beta\neq  0$. Then $\alpha=\delta$ and $\lambda_2=-\frac{\delta}{\beta}\lambda_3=-\frac{\gamma}{\alpha}\lambda_3$.\\
Case 2)-a)-1) $\lambda_3=0$. Then we get the case (iii).\\
Case 2)-a)-2) $\lambda_3\neq 0$.  Then we get the case (iv).\\
Case 2)-b) $\beta=0$. Then $\gamma=0$ and $\delta\lambda_3=0$, $\alpha\lambda_2=0$.\\
Case 2)-b)-1) $\alpha\neq 0$, $\delta\neq 0$. Then $\lambda_2=\lambda_3=0$. Solving (2.36), we get the case (v).\\
Case 2)-b)-2) $\alpha=0$, $\delta\neq 0$. Solving (2.36), we get $\delta=0$. This is a contradiction.\\
Case 2)-b)-3) $\alpha \neq 0$, $\delta=0$. Solving (2.36), we get $\alpha=0$. This is a contradiction.\\
\end{proof}
\indent By Theorem 4.2 in \cite{BO}, we have for $G_7$, there exists a pseudo-orthonormal basis $\{e_1,e_2,e_3\}$ with $e_3$ timelike such that the Lie
algebra of $G_7$ satisfies
\begin{equation}
[e_1,e_2]=-\alpha e_1-\beta e_2-\beta e_3,~~[e_1,e_3]=\alpha e_1+\beta e_2+\beta e_3,~~[e_2,e_3]=\gamma e_1+\delta e_2+\delta e_3,,~~\alpha+\delta\neq 0,~~\alpha\gamma=0.
\end{equation}
By (2.44) in \cite{Ca2}, we have for $G_7$
\begin{align}
&R_{1212}=\alpha\delta-\alpha^2-\beta\gamma-\frac{\gamma^2}{4},~~
R_{1313}=\alpha\delta-\alpha^2-\beta\gamma+\frac{\gamma^2}{4},\\\notag
&R_{2323}=-\frac{3}{4}\gamma^2,~~
R_{1213}=\alpha^2-\alpha\delta+\beta\gamma,~~R_{1223}=0,~~R_{1323}=0.
\end{align}
By page 16 in \cite{Ca1}, we get for $G_7$,
\begin{align}
L_Vg=\left(\begin{array}{ccc}
-2\alpha(\lambda_2-\lambda_3)&\alpha\lambda_1-\beta\lambda_2+(\beta+\gamma)\lambda_3&-\alpha\lambda_1+(\beta-\gamma)\lambda_2-\beta\lambda_3\\
\alpha\lambda_1-\beta\lambda_2+(\beta+\gamma)\lambda_3&2\beta\lambda_1+2\delta\lambda_3&-2\beta\lambda_1-\delta\lambda_2-\delta\lambda_3\\
-\alpha\lambda_1+(\beta-\gamma)\lambda_2-\beta\lambda_3&-2\beta\lambda_1-\delta\lambda_2-\delta\lambda_3&2\beta\lambda_1+2\delta\lambda_2
\end{array}\right).
\end{align}
By (2.5)(2.38)(2.39), we get that
$(G_7,V,g)$ is a left-invariant Riemann soliton if and only if
\begin{align}
\left\{\begin{array}{l}
\alpha\delta-\alpha^2-\beta\gamma-\frac{\gamma^2}{4}-(\beta\lambda_1+\delta\lambda_3)+\alpha(\lambda_2-\lambda_3)=-\lambda,\\
2(\alpha^2-\alpha\delta+\beta\gamma)+2\beta\lambda_1+\delta\lambda_2+\delta\lambda_3=0,\\
-\alpha\lambda_1+(\beta-\gamma)\lambda_2-\beta\lambda_3=0,\\
\alpha\delta-\alpha^2-\beta\gamma+\frac{\gamma^2}{4}-\beta\lambda_1-\delta\lambda_2-\alpha(\lambda_2-\lambda_3)=\lambda,\\
\alpha\lambda_1-\beta\lambda_2+(\beta+\gamma)\lambda_3=0,\\
-\frac{3}{4}\gamma^2-\delta\lambda_2+\delta\lambda_3=\lambda.\\
\end{array}\right.
\end{align}
\begin{thm}
$(G_7,V,g)$ is a left-invariant Riemann soliton if and only if \\
(i) $\alpha=0$, $\delta\neq 0$, $\beta=\gamma=0$, $\lambda_2=\lambda_3=\lambda=0$,\\
(ii) $\alpha=0$, $\delta\neq 0$, $\gamma=0$, $\beta\neq 0$, $\lambda_2=\lambda_3$, $\lambda=0$, $\lambda_1=-\frac{\delta}{\beta}\lambda_2$,\\
(iii) $\alpha\neq 0$, $\gamma=0$, $\alpha=\delta$, $\lambda_1=\lambda_2=\lambda_3=\lambda=0$.\\
\end{thm}
\begin{proof}
Case 1) $\alpha=0$. Then $\delta\neq 0$. By (2.40), we have
\begin{align}
\left\{\begin{array}{l}
-\beta\gamma-\frac{\gamma^2}{4}-(\beta\lambda_1+\delta\lambda_3)=-\lambda,\\
2\beta\gamma+2\beta\lambda_1+\delta\lambda_2+\delta\lambda_3=0,\\
(\beta-\gamma)\lambda_2-\beta\lambda_3=0,\\
-\beta\gamma+\frac{\gamma^2}{4}-\beta\lambda_1-\delta\lambda_2=\lambda,\\
-\beta\lambda_2+(\beta+\gamma)\lambda_3=0,\\
-\frac{3}{4}\gamma^2-\delta\lambda_2+\delta\lambda_3=\lambda.\\
\end{array}\right.
\end{align}
Case 1)-a) $\gamma\neq 0$. Then $\lambda_2=\lambda_3=0$ by the third equation and the fifth equation in (2.41). By (2.41), we have
\begin{align}
\left\{\begin{array}{l}
-\beta\gamma-\frac{\gamma^2}{4}-\beta\lambda_1=-\lambda,\\
\beta\gamma+\beta\lambda_1=0,\\
-\beta\gamma+\frac{\gamma^2}{4}-\beta\lambda_1=\lambda,\\
-\frac{3}{4}\gamma^2=\lambda.\\
\end{array}\right.
\end{align}
Case 1)-a)-1) $\beta=0$. By (2.42), we get $\gamma=0$. This is a contradiction.\\
Case 1)-a)-2) $\beta\neq 0$. By (2.42), we get $\lambda_1=-\gamma$ and $\gamma=0$. This is a contradiction.\\
Case 1)-b) $\gamma=0$. By (2.41), we have
\begin{align}
\left\{\begin{array}{l}
\beta\lambda_1+\delta\lambda_3=\lambda,\\
2\beta\lambda_1+\delta\lambda_2+\delta\lambda_3=0,\\
\beta(\lambda_2-\lambda_3)=0,\\
-\beta\lambda_1-\delta\lambda_2=\lambda,\\
-\delta\lambda_2+\delta\lambda_3=\lambda.\\
\end{array}\right.
\end{align}
Case 1)-b)-1) $\beta=0$. Solving (2.43), we get the case (i).\\
Case 1)-b)-2) $\beta\neq 0$. Solving (2.43), we get the case (ii).\\
Case 2) $\alpha\neq 0$. Then $\gamma=0$. By (2.40), we get
\begin{align}
\left\{\begin{array}{l}
\alpha\delta-\alpha^2-(\beta\lambda_1+\delta\lambda_3)+\alpha(\lambda_2-\lambda_3)=-\lambda,\\
2(\alpha^2-\alpha\delta)+2\beta\lambda_1+\delta\lambda_2+\delta\lambda_3=0,\\
-\alpha\lambda_1+\beta\lambda_2-\beta\lambda_3=0,\\
\alpha\delta-\alpha^2-\beta\lambda_1-\delta\lambda_2-\alpha(\lambda_2-\lambda_3)=\lambda,\\
-\delta\lambda_2+\delta\lambda_3=\lambda.\\
\end{array}\right.
\end{align}
By the third equation in (2.44), we have $\lambda_1=\frac{\beta}{\alpha}(\lambda_2-\lambda_3)$. By the first, the second and the fourth equations in (2.44), we get $\alpha=\delta$ and $2\beta\lambda_1+\delta\lambda_2+\delta\lambda_3=0$. By the fourth and the fifth equations in (2.44), we get
$\beta\lambda_1+\delta\lambda_2=0$ and $\lambda_2=\lambda_3$. Then by the fifth equation in (2.44), we get $\lambda=0$. So
$\lambda_1=\lambda_2=\lambda_3=0$. This is the case (iii).
\end{proof}
\vskip 1 true cm

\section{Acknowledgements}

The author was supported in part by NSFC No.11771070.

\vskip 1 true cm


\bigskip
\bigskip

\noindent {\footnotesize {\it Y. Wang} \\
{School of Mathematics and Statistics, Northeast Normal University, Changchun 130024, China}\\
{Email: wangy581@nenu.edu.cn}

\end{document}